\documentclass[12pt, reqno]{amsart}

\usepackage{mathrsfs}
\usepackage{amsfonts}
\usepackage[centertags]{amsmath}
\usepackage{amssymb,array}
\usepackage{amsthm}
\usepackage{stmaryrd}
\usepackage{graphicx}
\usepackage{caption}
\usepackage{ytableau}
\usepackage{enumerate,enumitem}
\usepackage{MnSymbol}
\SetLabelAlign{center}{\hfill #1\hfill}
\usepackage[textwidth=16cm, hmarginratio=1:1]{geometry}
\usepackage{tikz-cd, tikz}
\usepackage{rotating}
\usepackage[all,cmtip]{xy}
\usepackage[titletoc, title]{appendix}
\usepackage{amscd}
\usepackage[colorlinks]{hyperref}
\usepackage{float}
\usepackage{mathtools}
\usepackage{footnote}
\usepackage{comment}
\usepackage{dsfont} 
\hypersetup{
bookmarksnumbered,
pdfstartview={FitH},
breaklinks=true,
linkcolor=blue,
urlcolor=blue,
citecolor=blue,
bookmarksdepth=2
}

\theoremstyle{plain}
   
   \newtheorem{theorem}{Theorem}[section]

   \newtheorem{lemma}[theorem]{Lemma}
   
   \newtheorem{conjecture}[theorem]{Conjecture}
   
\theoremstyle{definition}
   \newtheorem{definition}[theorem]{Definition}
   
   \newtheorem{example}[theorem]{Example}

   \newtheorem{remark}[theorem]{Remark}

\numberwithin{equation}{section}

\newcommand{\CC}{{\mathbb {C}}}

\newcommand{\RR}{{\mathbb {R}}}
\newcommand{\ZZ}{{\mathbb {Z}}}

\newcommand{\ch}{{\operatorname{ch}}}

\newcommand{\SSYT}{{\rm SSYT}}

\DeclareMathOperator{\conv}{Conv}

\DeclareMathOperator{\Gr}{\mathrm{Gr}}
\DeclareMathOperator{\GL}{\mathrm{GL}}
\DeclareMathOperator{\Trop}{\mathrm{Trop}}
\DeclareMathOperator{\TGr}{\mathrm{TropGr}}

\DeclareMathOperator*{\wt}{wt}

\DeclareMathOperator{\Dr}{Dr}
\DeclareMathOperator{\TropGr}{TropGr}

\newcommand{\M}{\mathsf{M}}
\newcommand{\U}{\mathsf{U}}

\newcommand\scalemath[2]{\scalebox{#1}{\mbox{\ensuremath{\displaystyle #2}}}}

\newlength{\mysizetiny}
\newlength{\mysizesmall}
\newlength{\mysize}
\newlength{\mysizelarge}
\begin{document}

\title{From dual canonical bases to positroidal subdivisions}
\author{Jian-Rong Li and Ayush Kumar Tewari}
\address{Jian-Rong Li: Faculty of Mathematics, University of Vienna, Oskar-Morgenstern Platz 1, 1090 Vienna, Austria}
\email{lijr07@gmail.com}
\address{Ayush Kumar Tewari: Johann Radon Institute for Computational and Applied Mathematics (RICAM)
\\Austrian Academy of Sciences, Linz, Austria}
\email{ayushkumar.tewari@ricam.oeaw.ac.at}
\date{}

\begin{abstract}
The Grassmannian cluster algebra $\mathbb{C}[\Gr(k, n)]$ admits a distinguished basis known as the dual canonical basis, whose elements correspond to rectangular semi-standard Young tableaux with $k$ rows and with entries in $[n]$. We establish that each such tableau induces a positroidal subdivision of the hypersimplex $\Delta(k,n)$ via a map introduced by Speyer and Williams. For $\Gr(2,n)$, we prove that non-frozen prime tableaux correspond precisely to the coarsest positroidal subdivisions of $\Delta(2,n)$. Furthermore, we present computational evidence extending these results to $k>2$. In the process, we formulate a conjectural formula for the number of split positroidal subdivisions of $\Delta(k,n)$ for any $k \ge 2$ and explore the deep connections between the polyhedral combinatorics of $\Delta(k,n)$ and the dual canonical basis of $\mathbb{C}[\Gr(k, n)]$.
\end{abstract}


\maketitle

\section{Introduction} 

The hypersimplex $\Delta(k,n)$ is a classical polytope arising in various topics like algebraic combinatorics, discrete geometry, and matroid theory. It was notably studied by Gelfand, Goresky, MacPherson, and Serganova in their work on the combinatorial geometry of matroids \cite{gelfand1987combinatorial}. It is defined as the convex hull of all $0$--$1$ vectors in $\mathbb{R}^n$ with exactly $k$ entries equal to $1$. Subdivisions of $\Delta(k,n)$ respecting matroidal structure appear naturally in the study of tropical varieties and tropical prevarieties associated to the Grassmanian $\Gr(k,n)$ \cite{speyer2004tropical, speyer2008tropical}. This poset structure is captured geometrically by a polyhedral fan called the \emph{secondary fan} \cite{gelfand2009discriminants}, which encodes all regular subdivisions of the polytope. For matroidal subdivisions of $\Delta(k,n)$, the polyhedral fan encoding this poset is referred as the \emph{Dressian} $\Dr(k,n)$ \cite{speyer2008tropical}. 

The homogeneous coordinate ring $\mathbb{C}[\Gr(k,n)]$ admits a cluster algebra structure \cite{Sco06}. Canonical and dual canonical bases are important constructions in the representation theory of quantum groups by Lusztig \cite{Lus90} and Kashiwara \cite{Kas91}. The dual canonical basis of $\mathbb{C}[\Gr(k,n)]$ is a distinguished basis for $\mathbb{C}[\Gr(k,n)]$. In the Grassmannian setting, elements of this basis can be indexed by rectangular semi-standard Young tableaux with $k$ rows and entries in $[n]$, see \cite{CDFL}.

Speyer and Williams \cite{SW2005tropical} introduced a piecewise-linear map from the positive tropical Grassmannian to a subfan of the secondary fan of $\Delta(k,n)$, thereby linking tropical geometry with polyhedral subdivisions. We leverage this map to associate to each dual canonical basis tableau a positroidal subdivision of $\Delta(k,n)$, revealing a connection between dual canonical bases and polyhedral combinatorics, see Theorem \ref{thm:main_thm}.

We call a tableau prime if the corresponding dual canonical basis element in $\CC[\Gr(k,n)]$ is prime, that is, it cannot be written as the product of two dual canonical basis elements. Our primary result concerns $\Gr(2,n)$: we prove that non-frozen prime tableaux correspond precisely to the coarsest positroidal subdivisions of $\Delta(2,n)$, i.e., those subdivisions not admitting any nontrivial refinement. 
For $k \ge 3$, there are tableaux which are prime but the corresponding positroidal subdivisions of $\Delta(k,n)$ are not coarsest. We conjecture that one side of the result holds: if the positroidal subdivision of $\Delta(k,n)$ corresponding to a tableau $T$ is coarsest, then $T$ is a prime tableau. A related conjecture is Conjecture 6.1 in \cite{EL2024} in the case of ${\bf N}_{k,n}^{(1)}$ where it is conjectured that the facets of the Newton polytope ${\bf N}_{k,n}^{(1)}$ correspond to prime modules of the quantum affine algebra $U_q(\widehat{\mathfrak{sl}_k})$. 

A positroidal subdivision of $\Delta(k,n)$ is called a split if it has exactly two maximal cells \cite{herrmannsplitting}.
Finally, we propose a conjectural enumeration formula for the number of split positroidal subdivisions of $\Delta(k,n)$ for general $k \ge 2$: the number of split positroidal subdivision of $\Delta(k,n)$ is 
\begin{align*}
\frac{k-1}{2} n(n - k - 1),    
\end{align*}
see Conjecture \ref{conjecture:number of split positroidal subdivision}. Our results suggest deep interplay between positroidal subdivisions and the dual canonical basis.

\medskip

\noindent\textbf{Organization of the paper.}
In Section~\ref{sec:Matroidal and positroidal subdivisions}, we review hypersimplices, matroidal and positroidal subdivisions, and tropical Grassmannians. In Section \ref{sec:Dual canonical bases and semistandard Young tableaux}, we recall the relation between the dual canonical bases of Grassmannian cluster algebras and semistandard tableaux. In Section \ref{sec:From dual canonical bases to positroidal subdivisions}, we prove that every element in the dual canonical basis of a Grassmannian cluster algebra induces a positroidal subdivision of the hypersimplex $\Delta(k,n)$ under the map introduced by Speyer and Williams. In Section \ref{sec:Split positroidal subdivisions}, we study split positroidal subdivisions. 

\subsection*{Acknowledgements}

JRL is supported by the Austrian Science Fund (FWF): P-34602, Grant DOI: 10.55776/P34602, and PAT 9039323, Grant DOI 10.55776/PAT9039323.

%\section{Polyhedral definitions}

\section{Hypersimplices, matroidal and positroidal subdivisions, Tropical Grassmannians} \label{sec:Matroidal and positroidal subdivisions}
In this section, we review hypersimplices, matroidal and positroidal subdivisions, and tropical Grassmannians. 

\subsection{Hypersimplices, matroidal and positroidal subdivisions}

A \emph{matroid} of rank $k$ on the ground set $E$ is a nonempty collection $\M \subseteq \binom{|E|}{k}$ of $k$-element subsets of $E$, called \emph{bases} of $\M$, that satisfies the basis exchange axiom: 
For any $I , J \in \M$ and $a \in I$, there exists $b \in J$ such that $I \setminus \{ a \} \cup \{ b \} \in \M$.

A matroid is called \emph{representable} if it can be represented by columns of a matrix over some field $\mathbb{K}$. We index the columns of a $k \times n$ matrix by the set $[n]$. A \emph{positroid} $\mathsf{P}$ of rank $k$ is a matroid $\M$ that can be represented over $\mathbb{R}$ by a $k \times n$-matrix $A$ such that the maximal minor $p_I$ is non-negative for each $I \in \binom{[n]}{k}$. We refer the reader to the book by Oxley \cite{oxley2006matroid} for further details on matroids.  

A matroid intrinsically defines a polytope as follows \cite{edmonds2003submodular}.

\begin{definition}
Let $M$ be a matroid on a ground set $E$. Then the \emph{matroid polytope} $P_{M}$ is defined as 
\[ P_{M} = \conv\{e_{B} : B \text{ is a basis of } M \}  \]

where $\{e_{i}: i \in E \}$ is a basis of $\mathbb{R}^{E}$ and $e_{B} = e_{b_{1}} + \hdots + e_{b_{r}} $ for $B = \{b_{1} , \hdots , b_{r} \}$.
\end{definition}

Matroid polytopes can be characterized geometrically as follows \cite{gelfand1987combinatorial},

\begin{theorem}
A collection $\mathcal{B}$ of subsets of $[n]$ is the set of bases of a matroid if and only if every edge of the polytope
\[ P_{B} := \conv\{e_{B} : B \in \mathcal{B} \} \subseteq \mathbb{R}^{n}\]
is a translate of $e_{i} - e_{j}$ for some $i, j$.   
\end{theorem}

The \emph{hypersimplex} $\Delta(k,n)$ is defined as intersection of the unit cube $[0,1]^{n}$ with the hyperplane $\sum_{i \in [n]}x_{i} = k$. When we consider the matroid $M$ to be the \emph{uniform matroid} $\U_{k,n}$, then the corresponding matroid polytope is the hypersimplex, i.e., $P_{\U_{k,n}} = \Delta(k,n)$. 

\begin{definition}
For a polytope  $P \subseteq  \mathbb{R}^{d}$, a \emph{(polytopal) subdivision} $\Sigma = \bigg\{ P_1, P_2, \ldots, P_m\bigg\}$ is a collection of polytopes such that all of the following conditions are satisfied.
\begin{enumerate}
    \item The intersection $P_{ij} := P_i \cap P_j$, which can possibly be empty, is a face of both $P_i$ and $P_j$ for all $1 \leq i < j \leq m$. 
    \item The collection $\Sigma$ covers $P$ i.e. the union $\bigcup_{i=1}^m P_i = P$.
    \item For each $1 \leq i \leq m$, the vertex set $\mathrm{vert}(P_i) \subseteq \mathrm{vert}(P)$.
\end{enumerate}
The elements $P_i$ are referred as \emph{cells} of $\Sigma$. A polytopal subdivision is called a \emph{triangulation} if all $P_i$ are simplices. A polytopal subdivision is called \emph{matroidal} if all $P_i$ are matroid polytopes. Similarly, a polytopal subdivision is called  \emph{positroidal} if all $P_i$ are matroid polytopes of positroids.   
\end{definition}  

A subdivision $\Sigma$ of a polytope $P$ in $\mathbb{R}^{n}$ is said to be \emph{regular} if there exists a weight vector $w$ such that if the vertices of $P$ are lifted to heights provided by $w$ in $\mathbb{R}^{n+1}$ and subsequently the lower convex hull is projected back to $\mathbb{R}^{n}$, then the subdivision $\Sigma$ is retrieved. 

A \emph{split} subdivision is a subdivision with exactly two maximal cells \cite{herrmannsplitting}. It also known that any split subdivision is always regular \cite[Lemma 3.5]{herrmannsplitting}. Two splits $S_{1}$ and $S_{2}$ are said to be \emph{compatible} if the corresponding split hyperplanes do not intersect in the interior of the polytope. It is known that any two splits of the hypersimplex $\Delta(k,n)$ are compatible \cite[Corollary 5.6]{herrmannsplitting}.

We also recall the notions of polyhedral complexes and polyhedral fans.

\begin{definition}
A set $\mathcal{P}$ of polyhedra is called a \emph{polyhedral complex} if 
\begin{enumerate}
    \item $P \in \mathcal{P}$ and $F$ is a face of $P$, then it implies that $F \in \mathcal{P}$.
    \item The intersection of any two polyhedra  $P,Q \in \mathcal{P}$ is a face of both $P$ and $Q$.
\end{enumerate}
A \emph{polyhedral fan} in ${\mathbb{R}}^{n}$ is a polyhedral complex consisting of polyhedral cones.
\end{definition}

We refer the reader to \cite{ziegler2012lectures, de2010triangulations} for further details about polytopal combinatorics and polyhedral complexes.

\subsection{Tropical Grassmannians} 

The \emph{Grassmannian} $\Gr(k,n)$ is the smooth projective variety corresponding to the vanishing set of the \emph{Pl\"ucker ideal} $\mathcal{I}_{k,n}$. An element $V$ in the Grassmannian $\Gr(k,n)$ is a collection of $n$ vectors $v_{1}, \hdots, v_{n} \in \mathbb{K}^{k}$ spanning the space $\mathbb{K}^{k}$ modulo the  action of $\GL(k,n)$ on the vectors. Let $A$ be a $k \times n$-matrix consisting of column vectors $v_1, v_2, \ldots, v_n$ and $A$ is called the \emph{representative matrix} of $V$. This defines a matroid $\M_V$ whose bases are the $k$-subsets $I \subseteq [n]$ such that $\det_{I}(A) \neq 0$ where, $\det_{I}(A)$ denotes the determinant of the $k \times k$ submatrix of $A$ with the column set $I \in \binom{[n]}{k}$.

An element $V \in \Gr(k,n)$ is termed as \emph{totally non-negative} if it has a representative $k \times n$ matrix $A$ such that $\det_{I}(A) \geq 0$, for all $I \in \binom{[n]}{k}$. The set of all totally non-negative $V \in \Gr(k,n)$ is called the \emph{totally non-negative Grassmannian} $\Gr^{\geq 0}(k,n)$ and  $\Gr^{> 0}(k,n)$ is called the \emph{positive Grassmannian}.

Tropical geometry is the study of vanishing of polynomials over the tropical semiring $\mathbb{T} = \{ \mathbb{R} \cup \{ \infty \}, \text{max}, + \}$, where the locus of vanishing is the set of points where the maxima is achieved twice. Let $\mathcal{C}$ and $\mathcal{R}$ represent the Puiseux series over field of complex numbers and real numbers respectively and $\text{val}$ be the \emph{valuation} map. Let $\mathcal{R}^{+}$ represent those elements of $\mathcal{R}$ where the coefficient of the unique lowest term is real and positive. The tropicalisation of the variety $V(I)$ for and ideal $I \subset \mathcal{C}[x_{1}, \hdots, x_{n}]$ is defined as 

$$\Trop\bigl(V(I)\bigr)
=\overline{\bigl\{{\rm val}(x)\;\bigm|\;x\in V(I)\cap(\mathcal{C}^*)^n\bigr\}}. $$

Similarly, the \emph{positive part} of a tropical variety $\text{Trop}(V(I))$ is defined as 

$$\Trop^{+}\bigl(V(I)\bigr)
=\overline{\bigl\{{\rm val}(x)\;\bigm|\;x\in V(I)\cap(\mathcal{R}^{+})^n\bigr\}}.$$

The \emph{tropical Grassmannian} $\TGr(k,n)$ is the intersection of the tropical hypersurfaces $\Trop(f)$, where $f$ ranges over all elements of the \emph{Pl\"ucker ideal} $\mathcal{I}_{k,n}$ which is generated by the \emph{quadratic Pl\"ucker relations} \cite{maclagan2021introduction}. The \emph{Dressian} $\Dr(k,n)$ is the intersection of the tropical hypersurfaces $\Trop(f)$, where $f$ ranges over all three-term Pl\"ucker relations, that generate the \emph{Pl\"ucker ideal}. Similarly, the \emph{positive tropical Grassmannian} is the positive part of the tropical variety $\Trop^{+}\Gr(k,n)$ which is obtained as the intersection of the positive part of all tropical hypersurfaces $\Trop^{+}(f)$, where $f$ ranges over all elements of the Pl\"ucker ideal. The \emph{positive Dressian} $\Dr^{+}(k,n)$ is the intersection of the positive part of all tropical hypersurfaces $\Trop^{+}(f)$,
where $f$ ranges over all three-term Pl\"ucker relations. The underlying matroid for the definitions of the tropical Grassmannian and Dressian is the \emph{uniform matroid} $\U_{k,n}$ \cite{SW2005tropical}.

The Dressian $\Dr(k,n)$ enjoys an intricate relationship with the hypersimplex $\Delta(k,n)$ which can be described as follows,

\begin{theorem}[{\cite[Proposition 2.2]{speyer2008tropical}}]
A vector $w$ lies in the Dressian $\Dr(k,n)$ if and only if it induces a matroidal subdivision of the hypersimplex $\Delta(k,n)$.    
\end{theorem}

An analogous result is also true when restricted to the positive Dressian \cite{speyer2021positive, arkani2021positive}.

\begin{theorem}
A vector $w$ lies in the Dressian $\Dr^{+}(k,n)$ if and only if it induces a positroidal subdivision of the hypersimplex $\Delta(k,n)$.    
\end{theorem}

It is known for $k=2$, $\TGr(2,n) = \Dr(2,n)$ \cite{speyer2004tropical} as polyhedral fans, but for ($k \geq 3$ and $n \geq 6$),  $\TGr(k,n) \neq  \Dr(k,n)$ \cite{herrmann2009draw}. However, when restricted to just the positive part, we have the following result \cite{speyer2021positive,arkani2021positive}.

\begin{theorem}\label{thm:pos_trop_Grass_equals_pos_Dress}
The positive tropical Grassmannian $\Trop^{+}\Gr(k,n)$ equals the positive
Dressian $\Dr^{+}(k,n)$.
\end{theorem}

Theorem \ref{thm:pos_trop_Grass_equals_pos_Dress} is a powerful result and allows us to consider the positive tropical Grassmannian $\Trop^{+}\Gr(k,n)$ and the  positive
Dressian $\Dr^{+}(k,n)$ interchangeably. 

\section{Dual canonical bases and semistandard Young tableaux} \label{sec:Dual canonical bases and semistandard Young tableaux}
In this section, we recall dual canonical bases \cite{Lus90} and the relations between the dual canonical bases of Grassmannian cluster algebras with semistandard tableaux \cite{CDFL}. 

\subsection{Dual canonical basis}\label{subsec:dual canonical basis}
In this paper, we only need to use the case that Lie algebra $\mathfrak{g}$ is of type A. Let $\mathfrak{g}$ be a simple Lie algebra of type $A$ over $\mathbb{C}$ and consider its triangular decomposition $\mathfrak{g} = \mathfrak{n} \oplus \mathfrak{h} \oplus \mathfrak{n}_{-}$. Let $v$ be an indeterminate, and let
\[
U_v(\mathfrak{g}) = U_v(\mathfrak{n}) \otimes U_v(\mathfrak{h}) \otimes U_v(\mathfrak{n}_{-})
\]
be the corresponding Drinfeld-Jimbo quantum enveloping algebra over $\mathbb{C}(v)$, defined via a $v$-analogue of the Chevalley-Serre presentation of $U(\mathfrak{g})$. Lusztig \cite{Lus90} has defined a canonical basis $B$ of $U_v(\mathfrak{n})$. It is a linear basis of $U_v(\mathfrak{n})$ with nice positivity properties. The algebra $U_v(\mathfrak{n})$ is endowed with a distinguished scalar product. Denote by $B^*$ the basis of $U_v(\mathfrak{n})$ adjoint to the canonical basis $B$ with respect to this scalar product. Let $A_v(\mathfrak{n})$ be the graded dual of the vector space $U_v(\mathfrak{n})$. It can be endowed with a multiplication coming from the comultiplication of $U_v(\mathfrak{g})$ and it can be regarded as the quantum coordinate ring of the unipotent group $N$ with Lie algebra $\mathfrak{n}$ \cite{GLS13}. The basis $B^*$ naturally corresponds to a basis of $A_v(\mathfrak{n})$ and is referred to as the dual canonical basis of $A_v(\mathfrak{n})$. Upon specializing $v \to 1$, the basis $B^*$ yields a basis $\mathcal{B}$ for $\mathbb{C}[N]$, which is called the dual canonical basis of $\mathbb{C}[N]$. The basis $B^*$ can be identified with a basis of $A_v(\mathfrak{n})$ and $B^*$ is called the dual canonical basis of $A_v(\mathfrak{n})$. When $v \to 1$, $B^*$ become a basis $\mathcal{B}$ of $\mathbb{C}[N]$ and it is called the dual canonical basis of $\mathbb{C}[N]$. In \cite{HL15}, Hernandez and Leclerc showed that $\mathbb{C}[N]$ can be categorified using a certain category $\mathcal{C}_Q$ of representations of the quantum affine algebra $U_q(\widehat{\mathfrak{g}})$. There is a one to one correspondence between simple modules in $\mathcal{C}_Q$ and dual canonical basis elements in $\mathbb{C}[N]$. 

In type $A$ the variety $N$ is the variety consisting of all unipotent upper triangular matrices in $SL_n$. It is the big cell of the complete flag variety $\mathcal{F}\ell_{d_1,\ldots,d_r;n}$ when $r=n-1$ and $d_i=i$ for $i \in [1,n]$. General partial flag variety $\mathcal{F}\ell_{d_1,\ldots,d_r;n}$ also has a dual canonical basis and the basis elements can be identified with semistandard Young tableaux \cite{BL24, Li24}. When $r=1$, $\mathcal{F}\ell_{d_1,\ldots,d_r;n}$ is the Grassmannian variety $\Gr(k,n)$, where $k=d_1$. Denote by $\CC[\Gr(k,n)]$ the coordinate ring of the Grassmannian variety $\Gr(k,n)$. It is shown in \cite{Sco06} that $\CC[\Gr(k,n)]$ has a cluster algebra structure. The coordinate ring $\CC[\Gr(k,n)]$ can be categorified using a certain category $\mathcal{C}_{\ell}$ of representations of quantum affine algebra $U_q(\widehat{\mathfrak{sl}_k})$ \cite{HL10}, where $k+\ell+1=n$, and elements in the dual canonical basis of $\CC[\Gr(k,n)]$ are in one to one correspondence with semistandard tableaux with rectangular shapes and with $k$ rows \cite{CDFL}. We describe this correspondence in the next subsection.  

\subsection{Semistandard Young tableaux} \label{subsec:semistandard Young tableaux}
A semistandard Young tableau is a Young tableau with weakly increasing rows and strictly increasing columns, see \cite{fulton1997young}. For $k,n \in \ZZ_{\ge 1}$, we denote by ${\rm SSYT}(k, [n])$ the set of rectangular semistandard Young tableaux with $k$ rows and with entries in $[n] = \{1,\ldots, n\}$ (with arbitrarly many columns). The empty tableau is denoted by $\mathds{1}$. 

In the following, we recall a few definitions related to semistandard Young tableaux. For $S,T \in {\rm SSYT}(k, [n])$, let $S \cup T$ be the row-increasing tableau whose $i$th row is the union of the $i$th rows of $S$ and $T$ (as multisets), \cite{CDFL}. We call $S$ a factor of $T$, and write $S \subset T$, if the $i$th row of $S$ is contained in that of $T$ (as multisets), for $i \in [k]$. In this case, we define $\frac{T}{S}=S^{-1}T=TS^{-1}$ to be the row-increasing tableau whose $i$th row is obtained by removing that of of $S$ from that of $T$ (as multisets), for $i \in [k]$. A tableau $T \in {\rm SSYT}(k, [n])$ is trivial if each entry of $T$ is one less than the entry below it. For any $T \in {\rm SSYT}(k, [n])$, we  denote by $T_{\text{red}} \subset T$ the semistandard tableau obtained by removing a maximal trivial factor from $T$. For trivial $T$, one has $T_{\text{red}} = \mathds{1}$. Let ``$\sim$'' be the equivalence relation on $S, T \in {\rm SSYT}(k, [n])$ defined by: $S \sim T$ if and only if $S_{\text{red}} = T_{\text{red}}$. We denote by ${\rm SSYT}(k, [n],\sim)$ the set of $\sim$-equivalence classes. A one-column tableau is called a fundamental tableau if its entries are of the form $[j,j+k]\setminus \{a\}$ for some $a \in [j+1, j+k-1]$. For $i \in [k-1]$, $j \in [n-k]$, denote by $T_{i,j}$ the fundamental tableau with entries $[j,j+k]\setminus \{i+j\}$. A tableau $T$ is said to have small gaps if each of its columns is a fundamental tableau. Then any tableau in $\SSYT(k,[n])$ is $\sim$-equivalent to a unique small gap semistandard tableau. 

Denote by $\SSYT(k,[n],\sim)$ the set of equivalence classes of tableaux in $\SSYT(k,[n])$. We also call equivalence classes in $\SSYT(k,[n],\sim)$ tableaux. Denote by $\CC[\Gr(k,n,\sim)]$ the quotient of $\CC[\Gr(k,n)]$ by the ideal generated by $p_{i,i+1,\ldots,i+k-1}-1$, $i \in [1,n-k+1]$. 
It is shown in \cite{CDFL} that there is a one to one correspondence between dual canonical basis elements in $\CC[\Gr(k,n,\sim)]$ and tableaux in $\SSYT(k,[n],\sim)$. A map $\ch(T)$ is defined in Section 5 of \cite{CDFL} which sends a tableau $T$ in $\SSYT(k,[n],\sim)$ to an element in the dual canonical basis of $\CC[\Gr(k,n,\sim)]$. A map $\ch(T)$ is also defined in Section 5 of \cite{CDFL} for $T \in \SSYT(k,[n])$ which conjecturally sends $T$ to a dual canonical basis element in $\CC[\Gr(k,n)]$. We say that $T \in \SSYT(k,[n])$ is prime if there is no non-trivial tableaux $T', T'' \in \SSYT(k,[n])$ such that $\ch(T) = \ch(T')\ch(T'')$. 

\section{From dual canonical bases to positroidal subdivisions} \label{sec:From dual canonical bases to positroidal subdivisions}
In this section, we prove that every element in the dual canonical basis of a Grassmannian cluster algebra induces a positroidal subdivision of the hypersimplex $\Delta(k,n)$ under the map introduced by Speyer and Williams. 

The total number of fundamental tableaux in $\SSYT(k,[n])$ is $(k-1)(n-k)$. Let $T\in \SSYT(k,[n])$. Factor $T$ as union of fundamental tableaux $T \sim \cup_{i,j} T_{i,j}^{\cup c_{i,j}}$. Define $v_T = \sum_{i,j} c_{i,j} e_{i,j}$, where $e_{i,j}$'s form the standard basis of $\RR^{(k-1)(n-k)}$, see Section 6 in \cite{EL2024}. Denote by $F_{k,n}$ (see \cite{SW2005tropical}) the map from $\RR^{(k-1)(n-k)}$ to ${\rm Span}_{\RR}\{e^J: J \in \binom{[n]}{k}\} / L$, where $L$ is the lineality space, defined as follows
\begin{align*}
F_{k,n}(y) = \sum_{J \in \binom{[n]}{k}} P_{J}(y) e^J,
\end{align*}
where $P_{J}$ is the tropicalization of Pl\"{u}cker coordinate $p_J$ evaluating on the web matrix $W$. By definition, $F_{k,n}(v_T)=0$ for any trivial tableau $T$. We denote $\wt(T) = \wt_T = (P_J(v_T))_{J \in \binom{[n]}{k}}$.

\begin{theorem}\label{thm:main_thm}
For every $T \in \SSYT(k,[n])$,  $F_{k,n}(v_T)$ induces a positroidal subdivision of $\Delta(k,n)$.
\end{theorem}

\begin{proof}
We first realize that given $T\in \SSYT(k,[n])$, the image of the map $F_{k,n}$ spans the space ${\rm Span}_{\RR}\{e^J: J \in \binom{[n]}{k}\}$ and by \cite{SW2005tropical} we know that the Pl\"{u}cker coordinates $\{p_{J}\}_{J \in \binom{[n]}{k}}$ defined by elements in the web matrix lie in the positive Grassmannian $\Gr^{+}(k,n)$. By definition of the positive tropical Grassmannian, the corresponding tropicalizations of the Pl\"{u}cker coordinates $ 
{\rm wt}_{T} = (P_{J}(v_T))_{J \in \binom{[n]}{k}}$ lies in the positive tropical Grassmannian, i.e., ${\rm wt}_{T} \in \TropGr^{+}(k,n)$. By Theorem \ref{thm:pos_trop_Grass_equals_pos_Dress} we know that the positive tropical Grassmannian $\TropGr^{+}(k,n)$ is equal to the positive Dressian $\Dr^{+}(k,n)$ and therefore by the classification of points in the cones of the positive Dressian given in \cite[Theorem 1.1]{arkani2021positive} we conclude that the cone corresponding to ${\rm wt}_{T}$ in $\TropGr^{+}(k,n) = \Dr^{+}(k,n)$ corresponds to a positroidal subdivision of $\Delta(k,n)$. 
\end{proof}

Theorem \ref{thm:main_thm} shows that the weight $\wt_T$ of a tableau $T  \in \SSYT(k,[n])$ is the weight vector of the positroidal subdivision of $\Delta(k,n)$ induced by $T$.

We say that $T \in \SSYT(k,[n])$ has a frozen factor $S$ if $S \in \SSYT(k,[n])$ is a one-column tableau whose entries form a cyclic interval and $T=S\cup T'$, where $T' \in \SSYT(k,[n])$. Every $T \in \SSYT(k,[n])$ (resp. $T \in \SSYT(k,[n],\sim)$) can be written as $T=S^{\cup k}$ for some $S \in \SSYT(k,[n])$, where $k \in \ZZ_{\ge 1}$ is as larger as possible. We denote $\underline{T}=S$. 
\begin{conjecture} \label{conj:prime tableaux correspond to coarsest positroidal subdivision}
Let $T \in \SSYT(k,[n])$ be a tableau which has no frozen factors. If $F_{k,n}(v_T)$ defines a coarsest positroidal subdivision of $\Delta(k,n)$, then $\underline{T}$ is prime.
\end{conjecture} 

\begin{remark}
The converse of Conjecture \ref{conj:prime tableaux correspond to coarsest positroidal subdivision} is not true. In the case of $\Gr(3,8)$, the following $8$ prime tableaux 

\begin{align*}
\scalemath{0.7}{
\begin{ytableau}
1 & 2 & 3\\ 2 & 5 & 6\\ 4 & 7 & 8
\end{ytableau}, \
\begin{ytableau}
1 & 3 & 4 \\ 2 & 5 & 6 \\ 5 & 7 & 8
\end{ytableau}, \
\begin{ytableau}
1 & 3 & 4 \\ 2 & 6 & 7 \\ 5 & 8 & 8
\end{ytableau}, \
\begin{ytableau}
1 & 2 & 4 \\ 3 & 3 & 7 \\ 5 & 6 & 8
\end{ytableau}, \
\begin{ytableau}
1 & 1 & 2 \\ 3 & 4 & 5 \\ 6 & 7 & 8
\end{ytableau}, \
\begin{ytableau}
1 & 2 & 5 \\ 3 & 4 & 7 \\ 6 & 6 & 8
\end{ytableau}, \
\begin{ytableau}
1 & 2 & 3 \\ 4 & 4 & 5 \\ 6 & 7 & 8
\end{ytableau}, \
\begin{ytableau}
1 & 2 & 3 \\ 4 & 5 & 6 \\ 7 & 7 & 8
\end{ytableau}, }
\end{align*}
do not correspond to coarsest positroidal subdivisions of $\Delta(3,8)$. This is consistent with the fact that there are $120$ rays in the positive Tropical Grassmannian ${\TropGr(3,8)}^{+}$ but the cluster complex for the Grassmannian $\Gr(3,8)$ has 128 rays where the additional 8 rays correspond to $8$ cluster variables of degree $3$ \cite{bendle2024massively}. These $8$ cluster variables correspond to exactly the above $8$ tableaux. 
\end{remark}

\begin{remark}
A related conjecture is Conjecture 6.1 in \cite{EL2024} in the case of ${\bf N}_{k,n}^{(1)}$ where it is conjectured that the facets of the Newton polytope ${\bf N}_{k,n}^{(1)}$ correspond to prime modules of the quantum affine algebra $U_q(\widehat{\mathfrak{sl}_k})$. 
\end{remark}

\section{Split positroidal subdivisions} \label{sec:Split positroidal subdivisions}
In this section we restrict to a certain subclass of coarsest positroidal subdivisions, namely \emph{split positroidal subdivisions}. In the case when $k=2$ we prove that the cells of a split positroidal subdivision of $\Delta(2,n)$ can be completely described by the non-frozen prime tableaux under the map $F_{k,n}$.

\subsection{\texorpdfstring{Split positroidal subdivisions for $\Delta(2,n)$}{Split positroidal subdivisions for Delta(2,n)}}

We recall the following structural result about non-frozen prime tableau in $\SSYT(2, [n])$ \cite{CDHHHL}.

\begin{lemma}\label{lem:prime_calssific_2_n}
Every non-frozen prime tableau in $\SSYT(2, [n])$ is a one-column tableau with entries $\{a, b\}$, $a<b$. 
\end{lemma}

\begin{theorem}\label{thm:split_2_n}
Let $T = \scalemath{0.7}{\begin{ytableau}
    i \\ j
\end{ytableau} } \in \SSYT(2,[n])$, $i < j$, be a non-frozen prime tableau, then its image under the map $F_{k,n}$ induces a split positroidal subdivision of $\Delta(2,n)$.

\end{theorem}

\begin{proof}

By Theorem \ref{thm:main_thm} we know that for a non-frozen prime tableaux $T \in \SSYT(2,[n])$ its image under the map $F_{k,n}$ induces a positroidal subdivision $\Sigma$ of $\Delta(2,n)$. We proceed via contradiction and assume that $\Sigma$ is not a split positroidal subdivision. Since we know that all matroidal (positroidal) subdivisions of $\Delta(2,n)$ are obtained as common refinements of split matroidal (positroidal) subdivisions \cite[Proposition 7.11]{herrmannsplitting}, therefore there exists at least two or more split subdivisions $S_{1}, \hdots , S_{j}, (j \geq 2)$ such that $\Sigma$ is the common refinement of the split subdivisions $S_{j}$'s. Let $w_{\Sigma}$ be the weight inducing the subdivision $\Sigma$, then by the Split decomposition theorem \cite[Theorem 3.10]{herrmannsplitting}, up to scalar factors 

\[ w_{\Sigma} = w_{1} + w_{2} + \hdots + w_{j}.\]

As is already noted in \cite{SW2005tropical} that the map $F_{k,n}$ is a bijection, therefore the surjectivity of this map implies that there should exist tableaux $T_{1}, \hdots,  T_{j}$ each of which is at least of one column such that  $T = \cup_{i=1}^{m} T_{i}$ and the image of each $T_{i}, 1 \leq i \leq m$ under the map $F_{k,n}$, is $S_{i}$. But as $T$ is a non-frozen prime tableau in $\SSYT(2,[n])$, we know by Lemma \ref{lem:prime_calssific_2_n} that such tableau are one column and therefore such a decomposition is not possible and this gives us a contradiction, hence $\Sigma$ should be a split positroidal subdivision of $\Delta(2,n)$. 
\end{proof}

Notice that in the proof of Theorem \ref{thm:split_2_n}  we use the fact that non-frozen prime tableaux in $\SSYT(2,[n])$ are one column tableaux. We believe that the following could also be true.
\begin{conjecture}\label{conj:n-column_tab_k=2}
For every $T \in \SSYT(2,[n])$, such that $T$ has $m$ columns and let $(S_{1}, \ldots , S_{m})$ be the unique
unordered $m$-tuple of one-column tableaux that are pairwise weakly separated whose union is $T$. Then the weight $\wt_{T}$ of the positroidal subdivision corresponding to $T$ is given as
\begin{align*}
\wt\nolimits_{T} = \wt\nolimits_{S_{1}} + \hdots + \wt\nolimits_{S_{k}}.
\end{align*} 
\end{conjecture} 

We note that the existence of the unique
unordered $m$-tuple of one-column tableaux $(S_{1}, \ldots , S_{m})$ mentioned in Conjecture \ref{conj:n-column_tab_k=2} is shown in \cite[Lemma 3.1]{early2024classification}. We ask the question: for a collection of one-column tableaux $S_1, \ldots, S_m \in \SSYT(k,[n])$, if $S_1, \ldots, S_m$ are pairwise weakly separated, is $\wt(\cup_{i=1}^m S_i) = \sum_{i=1}^m \wt(S_i)$?

\begin{example}
We consider two non-frozen prime tableaux $T_{1} = \scalemath{0.6}{\begin{ytableau}
    3\\4\\7
\end{ytableau}}$ and $T_{2} = \scalemath{0.6}{\begin{ytableau}
    4\\6\\7
\end{ytableau}} \in \SSYT(3, [7])$. The weight vectors of $T_{1}$, $T_{2}$ (we use lexicographical ordering for $k$-subsets in $\binom{[n]}{k}$) are
\begin{align*}
& \wt\nolimits_{T_{1}} = (0, 0, 0, 0, 0, 0, 0, 0, 0, 0, 0, 0, 1, 1, 1, 0, 0, 0, 0, 0, 0, 0, 1, 1, 1, 0, 0, 0, 1, 1, 1, 1, 1, 1, 2), \\
& \wt\nolimits_{T_{2}}= (0, 0, 0, 0, 0, 0, 0, 0, 0, 0, 0, 0, 0, 0, 0, 0, 0, 0, 0, 0, 0, 0, 0, 0, 0, 0, 0, 0, 0, 0, 0, 0, 0, 0, 1).
\end{align*}  
Note that one-column tableaux $T_1$, $T_2$ are weakly separated. The weight vector of $T_1 \cup T_2$ is 
\begin{align*}
\wt\nolimits_{T_1 \cup T_2} & = (0, 0, 0, 0, 0, 0, 0, 0, 0, 0, 0, 0, 1, 1, 1, 0, 0, 0, 0, 0, 0, 0, 1, 1, 1, 0, 0, 0, 1, 1, 1, 1, 1, 1, 3) \\
& = \wt\nolimits_{T_{1}} + \wt\nolimits_{T_{2}}.
\end{align*}
\end{example}

\begin{example}
We consider the five non-frozen prime tableaux in $\SSYT(2,[5])$:
\begin{align*}
\scalemath{0.7}{
\begin{ytableau}
1  \\ 3
\end{ytableau}, \
\begin{ytableau}
1  \\ 4
\end{ytableau}, \
\begin{ytableau}
2 \\ 4
\end{ytableau}, \
\begin{ytableau}
2 \\ 5 
\end{ytableau}, \
\begin{ytableau}
3 \\ 5 
\end{ytableau}.}
\end{align*}

The image of these tableaux under the map $F_{k,n}$ provides weight vectors that induce split positroidal subdivisions and these weight vectors are given as follows:
\begin{align*}
w_{13} = \{0,0,0,0,1,0,0,0,0,0 \}, \\
w_{14} = \{0,0,0,0,1,1,0,1,0,0 \}, \\ 
w_{24} = \{0,0,0,0,0,0,0,1,0,0 \}, \\
w_{25} = \{0,0,0,0,0,0,0,1,1,1 \}, \\
w_{35} = \{0,0,0,0,0,0,0,0,0,1 \}.
\end{align*}

\begin{figure}
    \centering
    \includegraphics[scale=0.4]{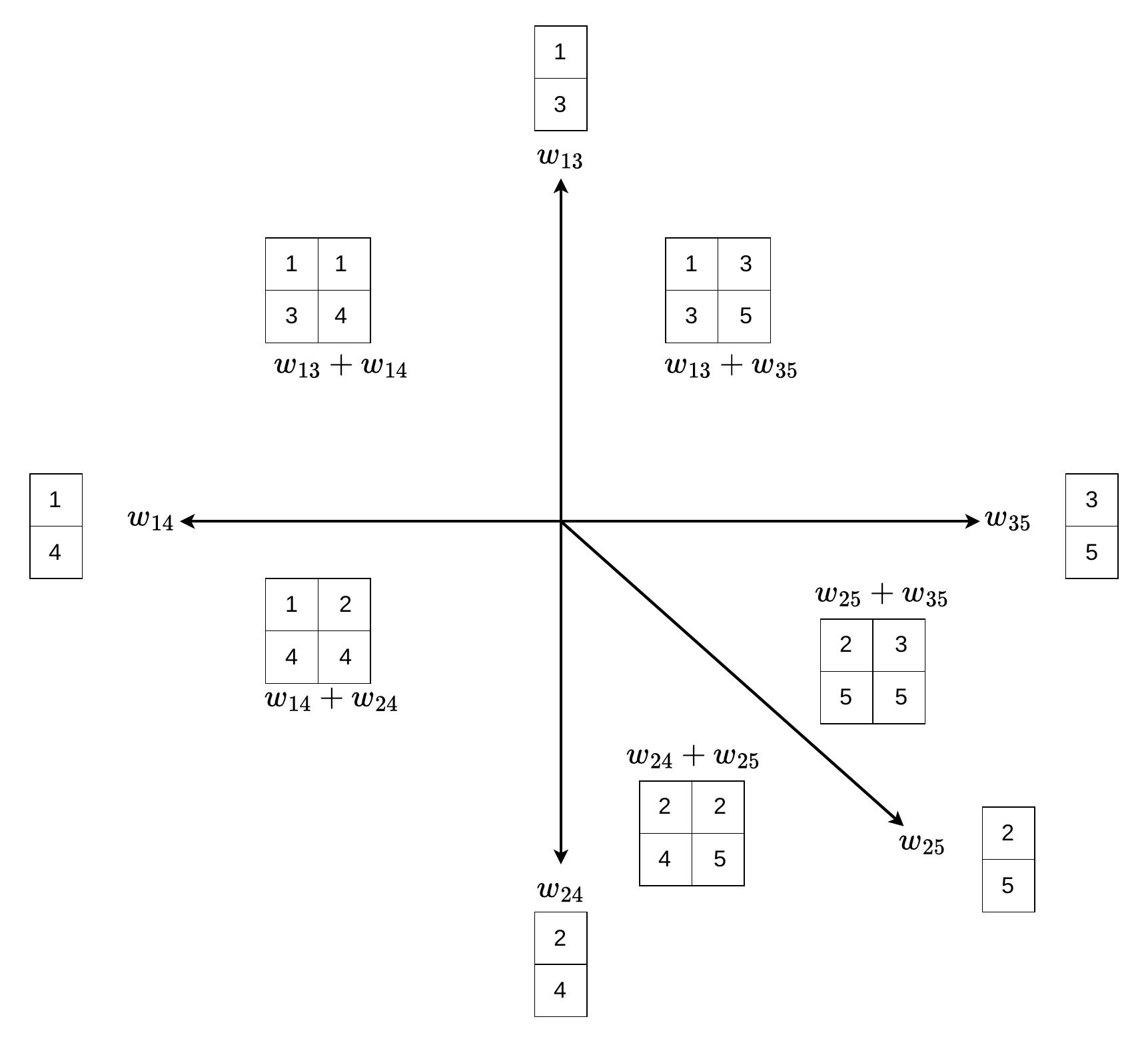}
    \caption{Positive Tropical Grassmannian $\TropGr(2,5)^{+}$ with the five rays and maximal cones described as common refinements of rays along with the corresponding non-frozen prime tableaux corresponding to each of the cones.}
    \label{fig:Pos_TropGr(2,5)}
\end{figure}

In Figure \ref{fig:Pos_TropGr(2,5)} the five rays of the positive tropical Grassmannian are shown along with the corresponding one column non-frozen prime tableaux that correspond to them. Also the maximal cones are listed along with the corresponding weight vectors, which as a consequence of being common refinements of splits, correspond to sum of weight vectors that induce the respective splits.

\end{example}

It is worth noting that maximal cells of a matroidal subdivision of $\Delta(2,n)$ can be described by the paths through the internal vertex of a phylogenetic tree on $n$ leaves \cite{kapranov16chow, speyer2004tropical}. Hence, we can restrict this finding to positroidal subdivisions as well, therefore split positroidal subdivisions correspond to phylogenetic trees with exactly one internal edge \cite[Proposition 13.1]{arkani2021positive}. We note then that the non frozen prime tableaux $T \in \SSYT(2,[n]), T = \scalemath{0.6}{ \begin{ytableau}
    i \\ j
\end{ytableau}}$, $i < j$ provides a canonical indexing of this phylogenetic tree illustrated in Figure \ref{fig:tree_label}. Therefore, the maximal cells of the split positroidal subdivision $\Sigma$ corresponding to $T = \{ij\}$ can be listed as paths in the following way:
\begin{align*}\label{eq:split_form}
& \Sigma = \{\left(\cup_{s=1}^{i} \{(s,t): t \in [i+1,j]   \}\right) \cup \left( \cup_{s=i+1}^{j}\{ (s,t): t \in [s+1,n] \} \right), \\
& \left(\cup_{s=1}^{i} \{(s,t): t \in [s+1,n] \}\right) \cup \left( \cup_{s=i+1}^{n}\{(s,t): t \in [j+1, n], s<t \} \right)\}.
\end{align*} 

\begin{figure}
    \centering
    \includegraphics[scale=0.55]{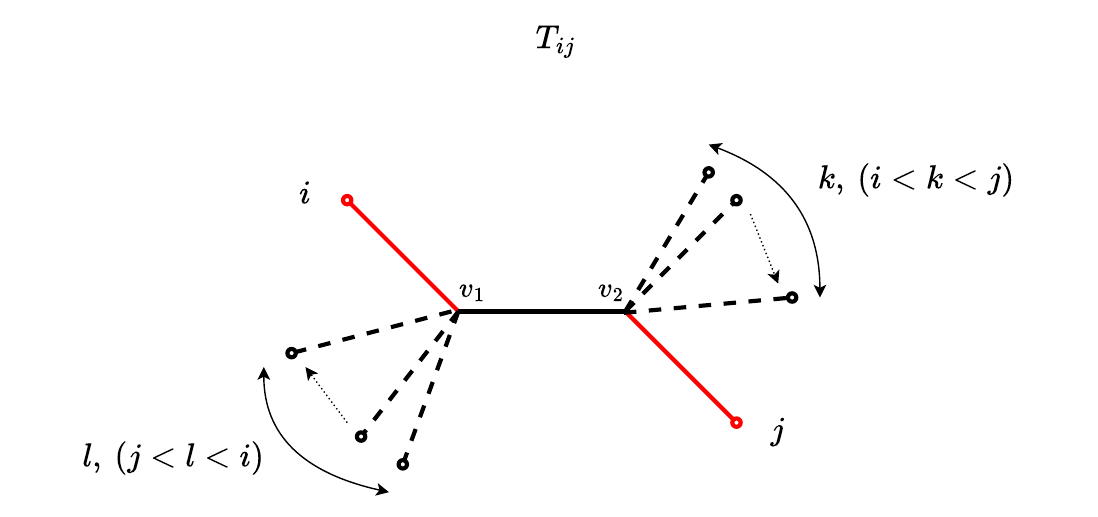}
    \caption{A phylogenetic tree corresponding to a split positroidal subdivision 
    $\Sigma$ of $\Delta(2,n)$ where the leaves $i$ and $j$ are the elements in the non-frozen prime tableaux $T = \{ ij\}$ that corresponds to $\Sigma$.}
    \label{fig:tree_label}
\end{figure}

\begin{figure}
    \centering
    \includegraphics[scale=0.5]{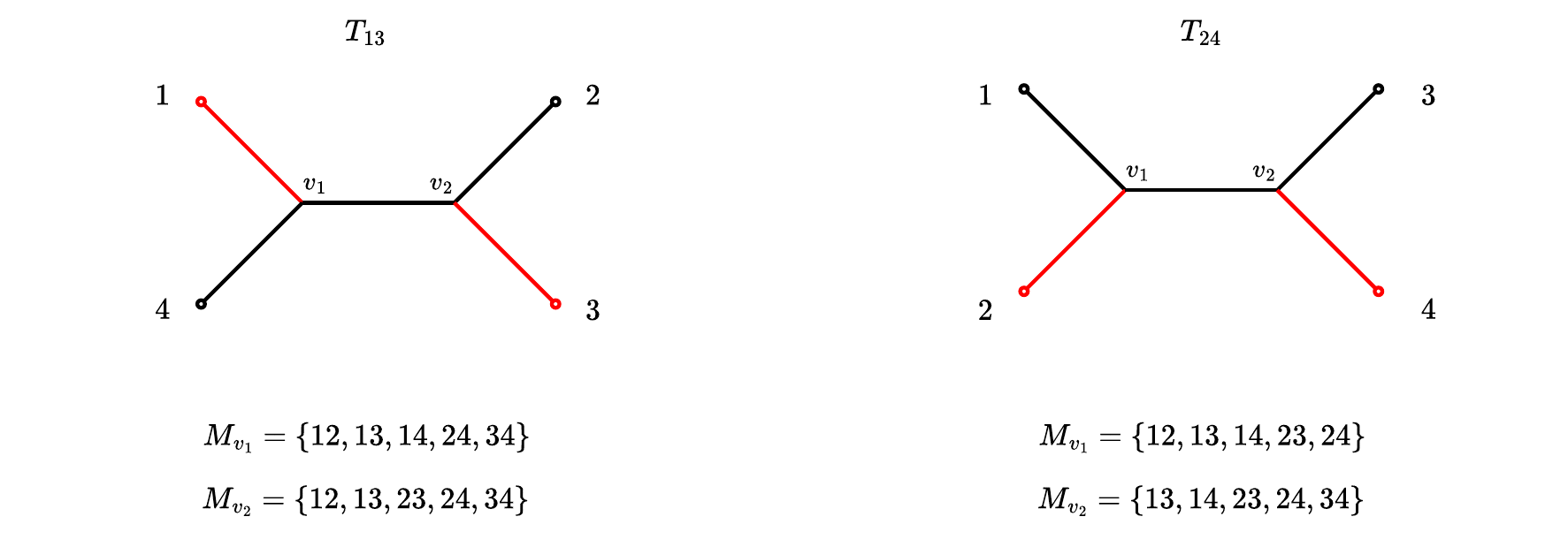}
    \caption{Phylogenetic trees that correspond to the two split positroidal subdivisions of $\Delta(2,4)$ where the maximal cells of the subdivision are indexed by the internal vertices $v_{1}$ and $v_{2}$ and the maximal cells can be described as paths passing through each internal vertex.  }
    \label{fig:split_tree_24}
\end{figure}

\begin{example}
In Figure \ref{fig:split_tree_24} and Figure \ref{fig:split_tree_25} we illustrate the two split positroidal subdivisions of $\Delta(2,4)$ and five split positroidal subdivisions of $\Delta(2,5)$ respectively along with the corresponding phylogenetic trees with the elements $i$ and $j$ that correspond to the entries in the non-frozen prime tableaux that provide a labelling of the leaves in the trees.      

\begin{figure}
    \centering
    \includegraphics[scale=0.45]{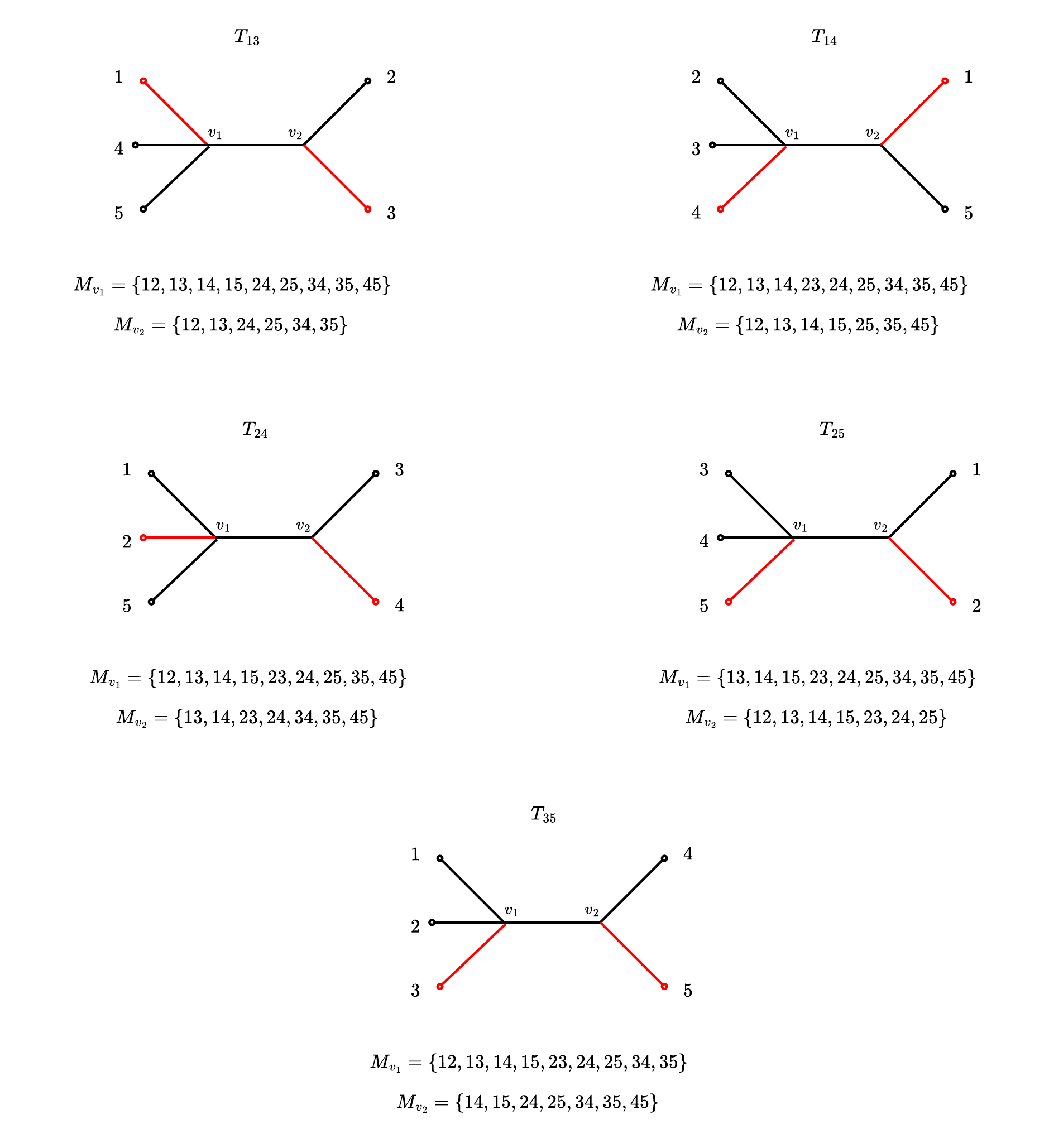}
    \caption{Phylogenetic trees that correspond to the five split positroidal subdivisions of $\Delta(2,5)$ where the maximal cells of the subdivision are indexed by the internal vertices $v_{1}$ and $v_{2}$ and the maximal cells can be described as paths passing through each internal vertex.  }
    \label{fig:split_tree_25}
\end{figure}

\end{example}

\subsection{\texorpdfstring{Split positroidal subdivisions for $\Delta(k,n)$ for $k \ge 2$}{Split positroidal subdivisions for Delta(k,n) for k >= 2}}

\begin{definition}
We say that a one-column tableau in $\SSYT(k,[n])$ has cyclical one gap if its entries are either of the form $[i_1, i_2] \cup [i_3, i_4]$ for some $i_1 \le i_2 <i_3-1$, $i_3 \le i_4$, $i_2-i_1 + i_4-i_3= k-2$, or of the form $[1,i_1]\cup [i_2, i_3] \cup [i_4,n]$ for some $1 \le i_1 < i_2-1$, $i_2 \le i_3 < i_4-1$, $i_4 \le n$.  
\end{definition}

\begin{conjecture} \label{conjecture:number of split positroidal subdivision}
Let $T \in \SSYT(k, [n])$, $k \geq 2$, be a non-frozen prime tableau. Then $F_{k,n}(v_T)$ is split if and only if $T$ is a one-column tableau and $T$ has one cyclical gap. For $k \ge 2$, the number of split positroidal subdivision of $\Delta(k,n)$ is $\frac{k-1}{2} n(n - k - 1)$. 
\end{conjecture}

We have verified the Conjecture \ref{conjecture:number of split positroidal subdivision} for $\Delta(3,7)$, $\Delta(3,8)$, $\Delta(3,9)$, $\Delta(4,8)$, $\Delta(4,9)$, $\Delta(4,10)$, $\Delta(4,11)$, $\Delta(5,10)$ and $\Delta(5,11)$ by using Sagemath \cite{sagemath} and  \texttt{polymake} \cite{gawrilow2000polymake}. All the files containing the code used for all these computations can be found at the following link

\begin{center}
    \url{https://github.com/Ayush-Tewari13/Positroidal_subdiv}
\end{center}

\begin{remark}
The one-column tableaux $[i_1, i_2] \cup [i_3, i_4]$ for some $i_1 \le i_2 <i_3-1$, $i_3 \le i_4$, $i_2-i_1 + i_4-i_3= k-2$, correspond to Kirillov-Reshetikhin modules of the quantum affine algebra $U_q(\widehat{\mathfrak{sl}_k})$.
\end{remark}

\bibliographystyle{siam}
\bibliography{biblio.bib}

\end{document}